\documentclass[12pt,russian]{article}
\usepackage{a4wide}
\usepackage[english]{babel}
\usepackage{amsfonts}
\usepackage{amssymb, amsthm}
\usepackage{amsmath}
\usepackage[cp1251]{inputenc}
\newtheorem{theorem}{Theorem}
\newtheorem*{corollary}{Corollary}

\begin{document}

\title{General Parity Result and Cycle-plus-Triangles Graphs}
\author{F. Petrov}
\maketitle

\let\thefootnote\relax\footnote{
St. Petersburg Department of
V.~A.~Steklov Institute of Mathematics of
the Russian Academy of Sciences, St. Petersburg State University.
E-mail: fedyapetrov@gmail.com.
Supported by Russian Scientific Foundation grant 14-11-00581.}

\begin{abstract}  We generalize a
parity result of Fleishner and Stiebitz that being combined with 
Alon--Tarsi polynomial method allowed them to prove that a
4-regular graph formed by a Hamiltonian
cycle and several disjoint triangles is always 3-choosable.
Also we present a modification of polynomial method and show how
it gives slightly more combinatorial information about colourings 
than direct application of Alon's Combinatorial Nullstellensatz. 
\end{abstract}

We start with the following parity theorem.

\begin{theorem}
Let $V=\sqcup_{i=1}^n V_i$ be a finite set partitioned onto disjoint subsets $V_i$
of odd sizes $|V_i|$. Let $G$ be a graph on a ground set $V$ such that each $V_i$
is independent set in $G$ and each bipartite subgraph induced on $V_i\sqcup V_j$
is Eulerian (i.e. all degrees are even). Consider the subsets $U\subset V$ such that
$|U\cap V_i|=1$ for all $i$ and subgraph induced on $U$ is Eulerian. Then the number
of such $U$ is odd.
\end{theorem}

\begin{proof} Consider ordered sequences of vertices 
$(x_1,\dots,x_n,y_1,\dots,y_n)$ so that for all $i=1,\dots,n$:

(i) $x_i\in V_i$;

(ii) $y_i\in \{x_1,\dots,x_n\}$;

(iii) either $y_i=x_i$ or $x_i$ and $y_i$ are joined by edge in $G$. Call it a special 
sequence.

We have to prove that number of such sequences is odd. Indeed, given $x_1,\dots,x_n$
fixed the number of ways to choose $y_1,\dots,y_n$ is odd if and only if the 
subgraph on $\{x_1,\dots,x_n\}$ is Eulerian. 

For any special sequence $\lambda=(x_1,\dots,y_n)$ we construct a 
directed graph $G(\lambda)$ on $\{1,2,\dots,n\}$:
draw directed edge from $i$ to $j\ne i$ if $y_i\in V_j$. This is a directed graph with 
outdegrees at most 1. Clearly $x_1,\dots,x_n$ and graph $G$
define $y_1,\dots,y_n$ (at most) uniquely. 
Further in the proof a cycle of length at least 3 is called long.
Denote by $\mathcal{L}$ the set of special sequences $\lambda$ for which 
$G(\lambda)$ has a long cycle.
We prove that $|\mathcal{L}|$ is even by constructing 
an involution without fixed points on $\mathcal{L}$.
It acts as follows. Choose a minimal (lexicographically) long cycle in $G$
and reverse its edges.  For example, if the minimal
cycle is formed by edges $2-3$, $3-9$, $9-2$, i.e. $y_2=x_3$, $y_3=x_9$, $y_9=x_2$,
we  replace $(y_2,y_3,y_9)$ from $(x_3,x_9,x_2)$ to $(x_9,x_2,x_3)$. No new
cycles appear, 
since any edge in $G$ 
may belong to at most one cycle. Hence this map is an involution
without fixed points, as desired.

Thus the  parity
of the number of special sequences is the same as the parity of the number
of special sequences $\lambda$ 
for which graph $G(\lambda)$ does not have long cycles. 
Number of special sequences with empty
$G(\lambda)$ equals $\prod |V_i|$, i.e. is odd. 
So, it suffices to prove that if a non-empty
directed graph $G$ on $\{1,\dots,n\}$
without long cycles is fixed, the number of special
sequences $\lambda$  with $G(\lambda)=G$ is even. 
It is almost obvious. Indeed, $G$ has a vertex
which has at most 1 neighbour, say, vertex $p$ is joined only with $q$
(by 1 or 2 edges). Fix all $x_i$ for $i\ne p$. The number of ways to choose
$x_p$ is the number of neighbors of $x_q$ in $V_p$, it is even number. 
\end{proof}

\begin{corollary} Given a circle $\gamma$.
Let $P_1,\dots,P_n$ be closed polygonal lines inscribed in $\gamma$, 
each having odd number of edges, without common
vertices. Then the total number of ways to choose edges $s_I$ of $P_i$,
$i=1,\dots,n$, so that each chosen edge intersects even number of 
other chosen edges, is odd.
\end{corollary}

The following corollary is the crucial 
parity theorem of \cite{FS}, originally proved by successive
modifications of the graph. 

\begin{corollary} Consider a 4-regular (multi)graph $G$ on the ground
set $V=\{x_1,\dots,x_{3n}\}$,
$x_{3n+1}=x_1$
(we identify vertices and abstract variables), 
which is a union of Hamiltonian cycle $x_1-x_2-\dots-x_{3n}-x_1$,
naturally considered as a regular $3n$-gon, and $n$ triangles
$a_i-b_i-c_i-a_i$, $i=1,\dots,n$. 
Consider the following Laurent polynomial 
$$
\Phi(x_1,\dots,x_{3n})=\prod_{i=1}^{3n} (1-x_{i+1}/x_i) \prod_{i=1}^n (1-a_i/b_i)(1-b_i/c_i)(1-c_i/a_i).
$$
Then the constant term $CT[\Phi]$ is congruent to 2 modulo 4.
\end{corollary}

\begin{proof} 
Start with expanding brackets in $\prod_{i=1}^{3n} (1-x_{i+1}/x_i)$.
We get monomials with each variable in a power 0 or $\pm 1$, and
powers $+1$ and $-1$ alternate. Coefficient of each such
monomial is $\pm 1$.  

Now consider triangles. We have $(1-a/b)(1-b/c)(1-c/a)=a/c+c/b+b/a-c/a-b/c-a/b$. 
That is, for any triangle we should take one vertex in power 1, another 
in power -1, third in power 0. Draw an arrow $a\rightarrow b$ if we choose 
$-a/b$, and so on. Additionally colour  $a$ in black and $b$ in white.
So, we have one black-to-white arrow for each triangle. 
Product of corresponding
multiples may cancel (i.e. give a constant product)
with the unique multiple  arising from $\prod_{i=1}^{3n} (1-x_{i+1}/x_i)$.
This happens if only if
black and white vertices alternate.
This in turn happens exactly when each chosen arrow intersects even
number of other chosen arrows. And for given $n$ not-oriented edges
there exist exactly two ways to draw arrows on them in such a way that
black and white vertices alternate. These two ways give a total amount +2
or -2 to the constant term of $\Phi$ (since they are obtained
one from another by changing summand in all $6n$ brackets).
Now we just use the previous corollary and
conclude that half of the constant term of $\Phi$ is odd.
\end{proof}

Now we write down the formula for the
central coefficient of $\Phi $ via the values of $\Phi$ on a grid.
Choose sets $A_1,\dots,A_{3n}$ of cardinality 3 in $K\setminus \{0\}$ for
some field
$K$ (we use only $K=\mathbb{R}$, but it is possible that other fields may be  useful
for other goals.)  Define a function
$\varphi_i$ on the set $A_i$. If, say, $A_i=\{u,v,w\}$ we put
$\varphi_i(u)=\frac{vw}{(u-v)(u-w)}$ and so on. Then 
\begin{equation}\label{auxiliary}
\sum_{x\in A_i} \varphi_i(x) x^d=\begin{cases}1 &\mbox{if } d=0 \\ 
0 & \mbox{if } d=1 \mbox{ or } d=2.\end{cases}
\end{equation}
The formula for the constant term is
\begin{equation}\label{main}
CT[\Phi]=\sum_{x_i\in A_i} \Phi(x_1,\dots,x_{3n}) \cdot
\prod_{i=1}^{3n}\varphi_i(x_i).
\end{equation}
Indeed, expand $\Phi(x_1,\dots,x_{3n})$. For each monomial term $\prod x_i^{d_i}$ 
in $\Phi$ the sum of products is a product of sums:
$$
\sum_{x_i\in A_i} \prod_{i=1}^{3n} x_i^{d_i}\cdot \prod_{i=1}^{3n} \varphi_i (x_i)=\prod_{i=1}^{3n} \left( \sum_{x\in A_i} \varphi_i(x) x^{d_i}\right)
$$
Equation \eqref{auxiliary} yields that this product equals 1 if
$d_1=\dots =d_{3n}=0$ (for the constant term) and equals 0 if $d_i\in \{1,2\}$ for at least one index $i$ (it happens
for each non-constant term of $\Phi$, that is seen from the formula for $\Phi$). This gives \eqref{main}.

Immediate corollary of \eqref{main} is that $\Phi$ can not vanish on $\prod A_i$. In other words, graph
$G$ is 3-choosable.

Now consider genuine colourings of $G$. Assume that $G$ is properly 3-coloured,
we identify colours with three real numbers $u,v,w$. Clearly vertices of any triangle
have different colours, therefore there are $n$ vertices of each colour. Denote
by $U$, $V$, $W$ number of $vw$-edges, $uw$-edges, $uv$-edges respectively
in the Hamiltonian cycle. Then $U+V$ is twice more than the number of $w$-vertices, i.e.
$U+V=2n$, analogously $V+W=U+W=2n$, hence $U=V=W=n$. 

Now
we take $A_1=A_2=\dots A_{3n}=A=\{u,v,w\}$ and apply formula \eqref{main}. Consider non-zero summand
in RHS of \eqref{main}, it corresponds to some  proper colouring. We have
\begin{align*}
\prod_{i=1}^{3n} \varphi_i(x_i)=&(-1)^{3n}\frac{(uvw)^{2n}}{(u-v)^{2n}(v-w)^{2n}(w-u)^{2n}}\\
\prod_{i=1}^n (1-a_i/b_i)(1-b_i/c_i)(1-c_i/a_i)=&\pm \frac{(u-v)^{n}(v-w)^{n}(w-u)^{n}}{(uvw)^{n}}\\
\prod_{i=1}^{3n} (1-x_{i+1}/x_i)&=\pm \frac{(u-v)^{n}(v-w)^{n}(w-u)^{n}}{(uvw)^{n}}.
\end{align*}
Totally 
$$
\Phi(x_1\dots,x_{3n}) \prod_{i=1}^{3n} \varphi_i(x_i)=\pm 1.
$$
Let's see what happens if we rename colours. 
If we simultaneously replace, say, colours $u$ and $v$, we totally change sign of
$2n$ or $6n$ multiples, hence sign of the product does not change. Therefore we may partition all non-zero summands
in the RHS of \eqref{main} onto 6-tuples with the same value of summands, and the number of 6-tuples equals the number 
of essentially different 3-colourings of $G$ (permutation of colours gives the same colouring). Hence we have proved 

\begin{theorem} Number of essentially different 3-colourings of the cycle-plus-triangles graph $G$ is odd.
\end{theorem}

\begin{corollary}
There exists a proper colouring of $G$ in 3 colours white, blue and red such that blue and red vertices form
a connected graph.
\end{corollary}

\begin{proof}
Assume the contrary. Then for given $n$ white vertices, other $2n$ vertices have $r\ge 2$ connected components.
We may interchange blue and red colour in each component by all possible $2^r$  ways. Therefore
total number of white-blue-red colourings is divisible by 4, on the other hand, it is 6 times more than the odd
number of essentially different colourings. The contradiction.
\end{proof}

Above modification of polynomial method is essentially the same as proposed in \cite{L,KP},
the only difference is that we apply to it Laurent polynomials directly, without making polynomials
from them. 
In \cite{FS} the authors used the method of \cite{AT}, which was later explained as application 
of Combinatorial Nullstellensatz  in the main 
survey by Alon \cite{Alon}. 

I am grateful to Vladislav Volkov for fruitful discussions.

\end{document}